\documentclass[11pt]{article}
\usepackage{amsmath,amsthm,verbatim,amssymb,amsfonts,amscd, graphicx}
\usepackage{graphics}
\usepackage{hyperref}
\usepackage{authblk}
\usepackage{lscape}
\usepackage{multirow}
\usepackage{float}
\theoremstyle{plain}
\numberwithin{equation}{section}
\newtheorem{theorem}{Theorem}[section]

\newtheorem{example}{Example}[section]
\newtheorem{definition}{Definition}[section]

\date{}

\title{\textbf{T-Controllability of Evolution Systems having Non-instantaneous Impulses}}
\author{$^*$Vishant Shah$^1$\thanks{ Corrorsponding Author}, Kalpesh Bharvad$^2$, Dipsha Bhadrecha$^3$, Aashi Maratha$^3$, Jaita Sharma$^4$, Dhanesh Patel$^5$\\
	Department of Applied Mathematics, The M. S. University of Baroda, Vadodara, India-390001\\

Email: vishantmsu83@gmail.com, kalpeshbharwad291@gmail.com, dipubhadrecha@gmail.com, aashimaratha@gmail.com, jaita.sharma-appmath@msubaroda.ac.in, dhanesh.patel-appmath@msubaroda.ac.in}

\begin{document}
	\maketitle
\begin{abstract}
In this manuscript, we have considered the system governed by a non-instantaneous impulsive dynamical system of integer ordered with classical and non-local conditions and derived sufficient conditions for the trajectory controllability of the system on the Banach space. The conditions were obtained through the concept of semi-group properties of operators and Gronwall's inequality. Finally, illustrations with the classical and non-local conditions were also added to validate the derived conditions.
\end{abstract}


\section{Introduction}
Impulsive differential equations play a very important role in studying the behavior of the phenomenon having abrupt changes in the physical problems. If the changes are at a particular time moment then it is called instantaneous impulsive differential equations and the changes are over the small intervals then it is called non-instantaneous impulsive differential equations. There is a wide range of applications of these impulsive evolution equations in all the fields of sciences namely, physical sciences, biological science, environmental sciences. These applications are found in the monograph \cite{Laxmi,Samko,Agarwal} and research articles \cite{Bainov,Erbe,Kirane,Joelianto,Dishlieva} and references their in.Qualitative properties like existence and uniqueness of solution and continuity of the solutions of instantaneous and non-instantaneous differential or integro-differential or evolution equations are found in research articles with initial conditions are found in research articles \cite{Milman,Rogovchenko,Rogovchenko1,Liu,Yang,Xiang,Anguraj,Zhang,Sattayatham,Liang,Fan,Wen,Tang,Chen,Chen1,Chen2,Cardinali,Shah,Hernandez,Sood} and references their in.

Controlllability is one of the important qualitative property of the systems in which, one has to find the controllers which drives the system from given initial state to desired final state or nearer to to the desired final state. The notion of the modern control theory for the finite dimensional linear system was first introduced by Kalmann. He used the concept of state space analysis to find the controller for the finite dimensional linear systems. Study of complete or exact cotrollability for the linear, semi-linear and non-linear finite and infinite dimensional systems using various techniques of operator theory is found in the monographs \cite{Russel,SED,BRW} and  articles \cite{Joshi,Klamka,Li, George,Klamka1,Klamka2} and reference their in. In the concept of complete or exact controllability for the system one has to find the controller which steers the system from given initial state to desire final state. But this type of controller may not be cost effective. To overcome this situation one has to find a controller which steers the system from given initial state to desired final state according to our requirements and study of finding a controller which steers the system to desired final state is called Trajectory controllability. The concept of trajectory controllability for a finite dimensional system was introduced by George \cite{George1} and extended to infinite dimensional system by Chalishajar, et. al. \cite{Chalishajar} using the concept of operator semi-group and fixed point theorems. The notion of trajectory controllability of semi-linear system with non-instantaneous impulses with classical and non-local conditions using the operator semi-group and Gronwall's inequality found in the article \cite{Vishant}. 

This manuscript established the trajectory controllability of the non-instantaneous system:
\begin{equation}
	\begin{aligned}
		x'(\tau)&=A(\tau)x(\tau)+F(\tau,x(\tau))+W_1(\tau), &\quad \tau \in [s_k,t_k+1),&\ for\ all\  k=0,1,2\cdots,p\\
		x'(\tau)&=A(\tau)x(\tau)+G(\tau,x(\tau))+W_2(\tau),&\quad \tau \in [t_k, s_k),&\ for\ all\ k=1,2,\cdots, p,  
	\end{aligned}\label{intro1}
\end{equation}
with local condition $x(0)=x_0$ and non-local condition $x(0)= x_0-h(x)$. 

The structure of the article is as follows:
\begin{itemize}
    \item [1.] Definitions of various controllability are discussed in section \eqref{prelim}.
    \item [2.] Section \eqref{classical} discussed the trajectory controlllability for the instantaneous impulsive system with classical conditions.
    \item[3.] Trajectory controllability for the instantaneous system with non-local conditions is discussed in the section \eqref{Nonlocal}.
    \item[4.] Finally, conclusion for the work is found in section \eqref{conclusion}
\end{itemize}

\section{Preliminaries}\label{prelim}
This section discusses definitions and prepositions to established trajectory controllability of the system governed by non-instantaneous impulsive evolution equation with classical as well nonlocal conditions.
\begin{definition}
The system \eqref{intro1} is completely controllable on the interval $\mathcal{J}=[0,T_0]$ if for any $x_0, x_1\in \mathcal{X}$, if there exist a control $W(\cdot)$n in $\mathcal{U}$ (control space) steers the system from $x_0$ at $\tau=0$ to $x_1$ at $\tau=T_0$. 	
\end{definition}
In the definition of complete controllablity, there is no information of the path or trajectory on which the given system to be driven. Sometimes this leads to high cost So to overcome this situation we select the path or trajectory (having minimum cost) under which control system drives from $x_0$ to $x_1$ over the interval $\mathcal{J}$. Searching of controller $W(\cdot)$ in a way that the system drive from $x_0$ to $x_1$ over the interval is called trajectory controllability of the system. Therefore, trajectory controllability of the system is strongest amongst all types of controllability.
\begin{definition}
Let, $\mathcal{C}_{\mathcal{T}}$ be the set of all trajectories under which the system \eqref{intro1} drives from $x_0$ to $x_1$ over the interval $\mathcal{J}$. The system \eqref{intro1} is trajectory controllable if for any $z\in \mathcal{C}_{\mathcal{T}}$, there is a controller $W(\cdot)\in \mathcal{U}$ such that state of the system $x(\tau)$ drives on prescribed trajectory $z(\tau)$. This means $x(\tau)=z(\tau)$ a.e. over the interval $\mathcal{J}$.    	
\end{definition}
	  
\section{T-controllability with classical conditions} \label{classical}
Consider the system governed by the non-instantaneous impulsive evolution equation
\begin{equation}
	\begin{aligned}
		x'(\tau)&=A(\tau)x(\tau)+F(\tau,x(\tau))+W_1(\tau),\quad \tau \in [s_k,t_{k+1})\\
		x'(\tau)&=A(\tau)x(\tau)+G(\tau,x(\tau))+W_2(\tau), \quad \tau \in [t_k,s_k)\\
		x(0)&=x_0
	\end{aligned}\label{class1}
\end{equation}
over the interval $[0, T_0]$. Here, $x(\tau)$ is the state of the system lies in Banach space $\mathcal{X}$ at any time $\tau\in [0, T_0]$, $A(\tau)$ at any time $\tau$ is a linear operator on the Banach space $\mathcal{X}$, $F, G:[0,T_0]\times \mathcal{X}\rightarrow \mathcal{X}$ are nonlinear functions. 

To discuss trajectory controllability of the system governed by non-instantaneous impulsive evolution equation \eqref{class1}, we have following theorem:
\begin{theorem}
	If,
	\begin{itemize}
		\item [(A1)] Linear operator $A$ in the system \eqref{class1} infinitesimal generator of $C_0$ semigroup.
		\item[(A2)] The non-linear map $F:[0,T_0]\times \mathcal{X}\rightarrow \mathcal{X}$ is continuous such that there exist a non-decreasing function $l_F:\mathbb{R}_+\rightarrow \mathbb{R}+$ and positive real number $r_0$ satisfying
		\[||F(\tau,x_1)-F(\tau,x_2)||\leq l_F(r)||x_1-x_2||,\] for all $\tau \in [0,T_0]$, $x_1,x_2\in B_r(\mathcal{X})$ and $r\leq r_0$.
		\item[(A3)] The non-linear map $G:[0,T_0]\times \mathcal{X}\rightarrow \mathcal{X}$ is continuous such that there exist a non-decreasing function $l_G:\mathbb{R}_+\rightarrow \mathbb{R}+$ and positive real number $r_0$ satisfying
		\[||G(\tau,x_1)-G(\tau,x_2)||\leq l_G(r)||x_1-x_2||,\] for all $\tau \in [0,T_0]$, $x_1,x_2\in B_r(\mathcal{X})$ and $r\leq r_0$.
	\end{itemize}
Then, the system \eqref{class1} is trajectory controllable over the interval $[0,T_0]$.
\end{theorem}
 \begin{proof}
 Let $u(\tau)$ be any trajectory in $\mathcal{C}_{\mathcal{\tau}}$ satisfying $x(t^{+}_k)=u(t^{+}_k)$ along which the system \eqref{class1} steered from initial state $x_0$ at $\tau=0$ to desired final state $x_1$ at $\tau=T_0$.
 
 Over the interval $[0, t_1)$, the system \eqref{class1} becomes:
 \begin{equation}
     \begin{aligned}
      x'(\tau)&=A(\tau)x(\tau)+F(\tau,x(\tau))+W_1(\tau)\\
      x(0)&=x_0 
     \end{aligned}\label{class2}
 \end{equation}
 Consider \[ W_1(\tau)=u'(\tau)-A(\tau)-F(\tau,u(\tau))\] over the interval $[0,t_1)$ in and plugging it in the \eqref{class2} the system \eqref{class2} becomes
 \[x'(\tau)=A(\tau)x(\tau)+F(\tau,x(\tau))+u'(\tau)-A(\tau)-F(\tau,u(\tau))\] with initial condition $x(0)-u(0)=0$.\\
 Choosing variable $z=x-u$ the equation system reduces to
 \begin{equation}
     \begin{aligned}
      z'(\tau)&=A(\tau)z(\tau)+F(\tau,x(\tau))-F(\tau,u(\tau))\\
      z(0)&=0
     \end{aligned}\label{class3}
 \end{equation}
 and problem of trajectory controllability of the system \eqref{class2} is reduced to the solvability of the system \eqref{class3} over the interval $[0,t_1)$.
 The mild solution of the system \eqref{class3} is given by:
 \begin{equation}
     z(\tau)=\int^{\tau}_{0} \mathcal{T}(\tau-\zeta)[F(\zeta,x(\zeta))-F(\zeta,u(\zeta))]d\zeta \label{class4}
 \end{equation}
 where, $\mathcal{T}(\tau)$ is $C_0$ semi-group generated by linear operator $A$ satisfying $||\mathcal{T}(\tau)||\leq M$ for some positive number $M$.\\
 Therefore,
 \begin{equation*}
     \begin{aligned}
      ||z(\tau)||&\leq \int^{\tau}_{0} ||\mathcal{T}(\tau-\zeta)||\ ||[F(\zeta,x(\zeta))-F(\zeta,u(\zeta))]||d\zeta\\
      &\leq M \int^{\tau}_{0} l_F(r) ||x(\zeta)-u(\zeta)|| d\zeta\\
      &\leq M \int^{\tau}_{0} l_F(r) ||z(\zeta)||d\zeta
     \end{aligned}
 \end{equation*}
  and using Gronwall's inequality, we obtain $z(\tau)=0$ over the interval $[0,t_1)$. Hence, $x(\tau)=u(\tau)$ for all $\tau \in [0, t_1)$. Therefore, the system is trajectory controllable over the interval $[0,t_1)$.
  
  Over the interval $[t_1, s_1)$, the system becomes:
 \begin{equation}
 \begin{aligned}
   x'(\tau)&=A(\tau)x(\tau)+G(\tau,x(\tau))+W_2(\tau),\\
   x(t_1)&=u(t_1)
   \end{aligned}
   \label{class5}
 \end{equation}
 as the state of the system satisfies $x(t_k)=u(t_k)$.
  Plugging the control $W_2(\tau)=u'(\tau)-A(\tau)u(\tau)-G(\tau,x(\tau))$ over the interval $[t_1,s_1)$ in the system \eqref{class5} the system becomes:
 \[ x'(\tau)-u'(\tau)=G(\tau,x(\tau))-G(\tau,u(\tau))\]
 Choosing $z(\tau)=x(\tau)-u(\tau)$ we obtain
 \[ z'(\tau)= G(\tau,x(\tau))-G(\tau,u(\tau))\]
 as the value of the $z$ at $\tau=t_1$ is zero.
 Therefore, we have
 \begin{equation*}
 \begin{aligned}
||z(\tau)||&\leq \int^{\tau}_{t_1} ||\mathcal{T}(\tau-\zeta)||\ ||[G(\zeta,x(\zeta))-G(\zeta,u(\zeta))]||d\zeta\\
&\leq M \int^{\tau}_{t_1} l_G(r) ||x(\zeta)-u(\zeta)|| d\zeta\\
      &\leq M \int^{\tau}_{t_1} l_G(r) ||z(\zeta)||d\zeta
\end{aligned}
 \end{equation*}
 using (A3) and Gronwall's inequality we obtain $z(\tau)=0$ for all $\tau \in [t_1,s_1)$  Therefore, system \eqref{class1} is T-Controllable over the interval $[t_1,s_1)$. Moreover, $z(s_1)=0$ as $G$ is continuous. 
 
 Continuing this process over the interval $[s_k, t_{k+1})$ the system \eqref{class1} becomes:
 \begin{equation}
     \begin{aligned}
      x'(\tau)&=A(\tau)x(\tau)+F(\tau,x(\tau))+W_1(\tau)\\
      x(s_k)&=u(s_k) 
     \end{aligned}\label{class6}
 \end{equation}
 Choose the control over the interval $[s_k,t_{k+1}$ as:
 \[ W_1(\tau)=u'(\tau)-A(\tau)-F(\tau,u(\tau))\] and plugging it in the equation \eqref{class6} we get,
 \[x'(\tau)=A(\tau)x(\tau)+F(\tau,x(\tau))+u'(\tau)-A(\tau)-F(\tau,u(\tau))\]
 considering $z(\tau)=x(\tau)-u(\tau)$, above expression becomes:
 \begin{equation}
 \begin{aligned}
      z'(\tau)&=A(\tau)z(\tau)+F(\tau,x(\tau))-F(\tau,u(\tau))\\
      z(s_k)&=0,
     \end{aligned}\label{class7}
 \end{equation}
 Therefore
 \begin{equation*}
     \begin{aligned}
      ||z(\tau)||&\leq \int^{\tau}_{s_k} ||\mathcal{T}(\tau-\zeta)||\ ||[F(\zeta,x(\zeta))-F(\zeta,u(\zeta))]||d\zeta\\
      &\leq M \int^{\tau}_{0} l_F(r) ||x(\zeta)-u(\zeta)|| d\zeta\\
      &\leq M \int^{\tau}_{0} l_F(r) ||z(\zeta)||d\zeta
     \end{aligned}
 \end{equation*}
 and using Gronwall's inequality, we obtain $z(\tau)=0$ over the interval $[s_k,t_{k+1})$. Hence, $x(\tau)=u(\tau)$ for all $\tau \in [s_k,t_{k+1})$. Therefore, the system is T- controllable over the interval $[s_k,t_{k+1})$.
 
 Over the interval $[t_k,s_k)$ the system becomes:
 \begin{equation}
 \begin{aligned}
   x'(\tau)&=A(\tau)x(\tau)+G(\tau,x(\tau))+W_2(\tau),\\
   x(t_k)&=u(t_k)
   \end{aligned}
   \label{class8}
 \end{equation}
 Plugging the control $W_2(\tau)=u'(\tau)-A(\tau)u(\tau)-G(\tau,x(\tau))$ over the interval $[t_k,s_k)$ in the system \eqref{class8} the system becomes:
 \[ x'(\tau)-u'(\tau)=G(\tau,x(\tau))-G(\tau,u(\tau))\]
 Choosing $z(\tau)=x(\tau)-u(\tau)$ we obtain
 \[ z'(\tau)= G(\tau,x(\tau))-G(\tau,u(\tau))\]
 as the value of the $z$ at $\tau=t_k$ is zero.
 Therefore, we have
 \begin{equation*}
 \begin{aligned}
||z(\tau)||&\leq \int^{\tau}_{t_k} ||\mathcal{T}(\tau-\zeta)||\ ||[G(\zeta,x(\zeta))-G(\zeta,u(\zeta))]||d\zeta\\
&\leq M \int^{\tau}_{t_k} l_G(r) ||x(\zeta)-u(\zeta)|| d\zeta\\
      &\leq M \int^{\tau}_{t_k} l_G(r) ||z(\zeta)||d\zeta
\end{aligned}
 \end{equation*}
 using (A3) and Gronwall's inequality we obtain $z(\tau)=0$ for all $\tau \in [t_k,s_k)$  Therefore, system \eqref{class1} is T-Controllable over the interval $[t_k,s_k)$.
 
 Since, the system is T-controllable over the intervals $[0,t_1),[t_1,s_1), [s_k,t_{k+1})$ and $[t_k,s_k)$ for all $k$. Hence, the system is controllable over entire interval $[0,T_0]$. This completes the proof of the theorem.
  \end{proof} 
 \begin{example}
 Let, $\mathcal{X}=L^2([0,\pi],\mathbb{R})$ and consider the system governed by non-instantaneous impulsive evolution equation:
 \begin{equation}
     \begin{aligned}
      \frac{\partial H(\tau,\psi)}{\partial \tau}&=\partial^{2}_{\psi} H(\tau,\psi)+F(\tau, H(\tau,\psi))+w_1(\tau,\psi)&\quad \tau \in [0,1/3)\cup [2/3,1], \\
      \frac{\partial H(\tau,\psi)}{\partial \tau}&=\partial^{2}_{\psi} H(\tau,\psi)+G(\tau,H(\tau,\psi))+w_2(\tau,\psi)&\quad \tau \in [1/3,2/3),\\
      H(\tau,0)&=0 \quad H(\tau, \pi)=0 & \quad \tau>0,\\
      H(0, \psi)&=H_0(\psi) \quad & \quad 0<\psi<\pi,
     \end{aligned}\label{example1}
 \end{equation}
 over the interval $[0,1]$.\\
 Defining the operator on the space $\mathcal{X}$ as $A(\tau)=\partial^{2}_{\psi}$, $A(\tau)$ is the infinitesimal generator of the $C_0$ semigroup $\mathcal{T}(\tau)$. The representation of $\mathcal{T}(\tau)$ is \[\mathcal{T}(\tau)z=\sum^{\infty}_{m=0} \exp(\mu_m\tau)<z,\phi_m> \phi_m\]
 where, $\phi_m=\sqrt{2}sin(n\psi)$ for all $m=1,2,\cdots$ is the orthonormal basis corresponding to eigenvalue $\mu_m=-m^2$ of the operator $A$.
 
 With this concept the equation \eqref{example1} can be rewritten as and abstract equation on the space $\mathcal{X}$ as
 \begin{equation}
    \begin{aligned}
     x'(\tau)&=A(\tau)+F(\tau,x)+W_1(\tau) &\quad  \tau \in [0,1/3)\cup [2/3,1],\\
     x'(\tau)&=A(\tau)+F(\tau,x)+W_2(\tau) &\quad \tau \in [1/3,2/3),\\
     x(0)&=x_0, &\quad
     \end{aligned}\label{example2}
 \end{equation}
 where, $x(\tau)=H(\tau,\cdot), W_1(\tau)=w_1(\tau,\psi)$ and $W_2(\tau)=w_2(\tau,\psi)$.
 The system \eqref{example2} is trajectory controllable over the interval $[0,1]$ if the functions $F$ and $G$ satisfies the hypotheses of the theorem. 
 \end{example}
 
 \section{T-controllability with non-local conditions}\label{Nonlocal}
 Consider the system governed by the non-instantaneous impulsive evolution equation
\begin{equation}
	\begin{aligned}
		x'(\tau)&=A(\tau)x(\tau)+F(\tau,x(\tau))+W_1(\tau),\quad \tau \in [s_k,t_{k+1})\\
	x'(\tau)&=A(\tau)x(\tau)+G(\tau,x(\tau))+W_2(\tau), \quad \tau \in [t_k,s_k)\\
		x(0)&=h(x)
	\end{aligned}\label{nonlocal1}
\end{equation}
over the interval $[0, T_0]$. Here, $x(\tau)$ is the state of the system lies in Banach space $\mathcal{X}$ at any time $\tau\in [0, T_0]$, $A(\tau)$ at any time $\tau$ is a linear operator on the Banach space $\mathcal{X}$, $F, G:[0,T_0]\times \mathcal{X}\rightarrow \mathcal{X}$ are nonlinear functions and $h:\mathcal{X}\rightarrow \mathcal{X}$ is the operator representing the non-local conditions.
The mild solution of the equation \eqref{nonlocal1} is given by 
\begin{equation}
    x(\tau)= \begin{cases}\begin{aligned}
   &\mathcal{T}(\tau)h(x)+\int^{\tau}_{0}\mathcal{T}(\tau-\zeta) \big[F(\zeta,x(\zeta))+W_1(\zeta)\big] d\zeta, &\quad \tau \in [0,t_1)\\
   &\mathcal{T}(\tau)x(t_k)+\int^{\tau}_{t_k}\mathcal{T}(\tau-\zeta) \big[G(\zeta,x(\zeta))+W_2(\zeta)\big] d\zeta,, & \quad \tau \in [t_k,s_k)\\
   &\mathcal{T}(\tau)x(s_k)+\int^{\tau}_{s_k}\mathcal{T}(\tau-\zeta) \big[F(\zeta,x(\zeta))+W_1(\zeta)\big] d\zeta, &\quad \tau \in [s_k,t_{k+1}),\\
   \end{aligned}
    \end{cases}\label{nonlocal2}
\end{equation}
where, $\mathcal{T}(\tau)$ is semigroup generated by the linear operator $A(\tau)$.

The following theorem discusses the trajectory controllability of the system governed by the equation \eqref{nonlocal1}.
\begin{theorem}
If,
\begin{itemize}
		\item [(A1)] Linear operator $A$ in the system \eqref{class1} infinitesimal generator of $C_0$ semigroup.
		\item[(A2)] The non-linear map $F:[0,T_0]\times \mathcal{X}\rightarrow \mathcal{X}$ is continuous such that there exist a non-decreasing function $l_F:\mathbb{R}_+\rightarrow \mathbb{R}+$ and positive real number $r_0$ satisfying
		\[||F(\tau,x_1)-F(\tau,x_2)||\leq l_F(r)||x_1-x_2||\], for all $\tau \in [0,T_0]$, $x_1,x_2\in B_r(\mathcal{X})$ and $r\leq r_0$.
		\item[(A3)] The non-linear map $G:[0,T_0]\times \mathcal{X}\rightarrow \mathcal{X}$ is continuous such that there exist a non-decreasing function $l_G:\mathbb{R}_+\rightarrow \mathbb{R}+$ and positive real number $r_0$ satisfying
		\[||G(\tau,x_1)-G(\tau,x_2)||\leq l_G(r)||x_1-x_2||\], for all $\tau \in [0,T_0]$, $x_1,x_2\in B_r(\mathcal{X})$ and $r\leq r_0$.
		\item[(A4)] The function $h:\mathcal{X}\rightarrow \mathcal{X}$ is Lipchitz continuous with Lipchitz constant $0\leq l_h\leq 1$.
	\end{itemize}
Then, the system \eqref{nonlocal1} is trajectory controllable over the interval $[0,T_0]$.
\end{theorem}
\begin{proof}
Let $u(\tau)$ be any trajectory in $\mathcal{C}_{\mathcal{\tau}}$ satisfying $x(t^{+}_k)=u(t^{+}_k)$ along which the system \eqref{nonlocal1} steered from initial state $x(0)=h(x)$ at $\tau=0$ to desired final state $x_1$ at $\tau=T_0$.

Over the interval $[0, t_1)$, the system \eqref{nonlocal1} becomes:
 \begin{equation}
     \begin{aligned}
      x'(\tau)&=A(\tau)x(\tau)+F(\tau,x(\tau))+W(\tau)\\
      x(0)&=h(x) 
     \end{aligned}\label{nonlocal3}
 \end{equation}
 Consider \[ W(\tau)=u'(\tau)-A(\tau)-F(\tau,u(\tau))\] over the interval $[0,t_1)$ and plugging it in the system \eqref{nonlocal3}, the system becomes
 \[x'(\tau)=A(\tau)x(\tau)+F(\tau,x(\tau))+u'(\tau)-A(\tau)-F(\tau,u(\tau))\] with initial condition $x(0)-u(0)=h(x)-h(u)$.\\
 Choosing variable $z=x-u$ the equation system reduces to
 \begin{equation}
     \begin{aligned}
      z'(\tau)&=A(\tau)z(\tau)+F(\tau,x(\tau))-F(\tau,u(\tau))\\
      z(0)&=h(x)-h(u)
     \end{aligned}\label{nonlocal4}
 \end{equation}
 and problem of trajectory controllability of the system \eqref{nonlocal3} is reduced to the solvability of the system \eqref{nonlocal4} over the interval $[0,t_1)$.
 The mild solution of the system \eqref{nonlocal4} is given by:
 \begin{equation}
     z(\tau)=\mathcal{T}(\tau)[h(x)-h(u)]+\int^{\tau}_{0} \mathcal{T}(\tau-\zeta)[F(\zeta,x(\zeta))-F(\zeta,u(\zeta))]d\zeta \label{nonlocal5}
 \end{equation}
 where, $\mathcal{T}(\tau)$ is $C_0$ semigroup generated by linear operator $A$ satisfying $||\mathcal{T}(\tau)||\leq M$ for some positive number $M$.\\
 Therefore, 
 \begin{equation*}
     \begin{aligned}
      ||z(\tau)||&\leq ||\mathcal{T}(\tau)||||h(x)-h(u)||+\int^{\tau}_{0} ||\mathcal{T}(\tau-\zeta)||\ ||[F(\zeta,x(\zeta))-F(\zeta,u(\zeta))]||d\zeta\\
      &\leq Ml_h||x(\tau)-u(\tau)||+M \int^{\tau}_{0} l_F(r) ||x(\zeta)-u(\zeta)|| d\zeta\\
      &\leq Ml_h||z(\tau)||+ M \int^{\tau}_{0} l_F(r) ||z(\zeta)||d\zeta
     \end{aligned}
 \end{equation*}
 This implies
 \begin{equation*}
 ||z(\tau)||\leq \frac{Ml_F(r)}{1-Ml_h}\int^{\tau}_{0}||z(\zeta)||d\zeta 
 \end{equation*}
 Using Gronwall's inequality, we obtain $z(\tau)=0$ over the interval $[0,t_1)$. Hence, $x(\tau)=u(\tau)$ for all $\tau \in [0, t_1)$. Therefore, the system \eqref{nonlocal1} is trajectory controllable over the interval $[0,t_1)$.
 
 Over the interval $[t_1, s_1)$, the system becomes:
 \begin{equation}
 \begin{aligned}
   x'(\tau)&=A(\tau)x(\tau)+G(\tau,x(\tau))+W_2(\tau),\\
   x(t_1)&=u(t_1)
   \end{aligned}
   \label{nonlocal6}
 \end{equation}
 as the state of the system satisfies $x(t_k)=u(t_k)$.
  Plugging the control $W_2(\tau)=u'(\tau)-A(\tau)u(\tau)-G(\tau,x(\tau))$ over the interval $[t_1,s_1)$ in the system \eqref{nonlocal6} the system becomes:
 \[ x'(\tau)-u'(\tau)=G(\tau,x(\tau))-G(\tau,u(\tau))\]
 Choosing $z(\tau)=x(\tau)-u(\tau)$ we obtain
 \[ z'(\tau)= G(\tau,x(\tau))-G(\tau,u(\tau))\]
 as the value of the $z$ at $\tau=t_1$ is zero.
 Therefore, we have
 \begin{equation*}
 \begin{aligned}
||z(\tau)||&\leq \int^{\tau}_{t_1} ||\mathcal{T}(\tau-\zeta)||\ ||[G(\zeta,x(\zeta))-G(\zeta,u(\zeta))]||d\zeta\\
&\leq M \int^{\tau}_{t_1} l_G(r) ||x(\zeta)-u(\zeta)|| d\zeta\\
      &\leq M \int^{\tau}_{t_1} l_G(r) ||z(\zeta)||d\zeta
\end{aligned}
 \end{equation*}
 using (A3) and Gronwall's inequality we obtain $z(\tau)=0$ for all $\tau \in [t_1,s_1)$  Therefore, system \eqref{nonlocal1} is T-Controllable over the interval $[t_1,s_1)$. Moreover, $z(s_1)=0$ as $G$ is continuous. 
 
 Continuing this process over the interval $[s_k, t_{k+1})$ the system \eqref{nonlocal1} becomes:
 \begin{equation}
     \begin{aligned}
      x'(\tau)&=A(\tau)x(\tau)+F(\tau,x(\tau))+W_1(\tau)\\
      x(s_k)&=u(s_k) 
     \end{aligned}\label{nonlocal7}
 \end{equation}
 Choose the control over the interval $[s_k,t_{k+1})$ as:
 \[ W_1(\tau)=u'(\tau)-A(\tau)-F(\tau,u(\tau))\] and plugging it in the equation \eqref{nonlocal7} we get,
 \[x'(\tau)=A(\tau)x(\tau)+F(\tau,x(\tau))+u'(\tau)-A(\tau)-F(\tau,u(\tau))\]
 considering $z(\tau)=x(\tau)-u(\tau)$, above expression becomes:
 \begin{equation}
 \begin{aligned}
      z'(\tau)&=A(\tau)z(\tau)+F(\tau,x(\tau))-F(\tau,u(\tau))\\
      z(s_k)&=0,
     \end{aligned}\label{nonlocal8}
 \end{equation}
 Therefore
 \begin{equation*}
     \begin{aligned}
      ||z(\tau)||&\leq \int^{\tau}_{s_k} ||\mathcal{T}(\tau-\zeta)||\ ||[F(\zeta,x(\zeta))-F(\zeta,u(\zeta))]||d\zeta\\
      &\leq M \int^{\tau}_{0} l_F(r) ||x(\zeta)-u(\zeta)|| d\zeta\\
      &\leq M \int^{\tau}_{0} l_F(r) ||z(\zeta)||d\zeta
     \end{aligned}
 \end{equation*}
 and using Gronwall's inequality, we obtain $z(\tau)=0$ over the interval $[s_k,t_{k+1})$. Hence, $x(\tau)=u(\tau)$ for all $\tau \in [s_k,t_{k+1})$. Therefore, the system \eqref{nonlocal1}is T- controllable over the interval $[s_k,t_{k+1})$.
 
 Over the interval $[t_k,s_k)$ the system becomes:
 \begin{equation}
 \begin{aligned}
   x'(\tau)&=A(\tau)x(\tau)+G(\tau,x(\tau))+W_2(\tau),\\
   x(t_k)&=u(t_k)
   \end{aligned}
   \label{nonlocal9}
 \end{equation}
 Plugging the control $W_2(\tau)=u'(\tau)-A(\tau)u(\tau)-G(\tau,x(\tau))$ over the interval $[t_k,s_k)$ in the system \eqref{nonlocal9} the system becomes:
 \[ x'(\tau)-u'(\tau)=G(\tau,x(\tau))-G(\tau,u(\tau))\]
 Choosing $z(\tau)=x(\tau)-u(\tau)$ we obtain
 \[ z'(\tau)= G(\tau,x(\tau))-G(\tau,u(\tau))\]
 as the value of the $z$ at $\tau=t_k$ is zero.
 Therefore, we have
 \begin{equation*}
 \begin{aligned}
||z(\tau)||&\leq \int^{\tau}_{t_k} ||\mathcal{T}(\tau-\zeta)||\ ||[G(\zeta,x(\zeta))-G(\zeta,u(\zeta))]||d\zeta\\
&\leq M \int^{\tau}_{t_k} l_G(r) ||x(\zeta)-u(\zeta)|| d\zeta\\
      &\leq M \int^{\tau}_{t_k} l_G(r) ||z(\zeta)||d\zeta
\end{aligned}
 \end{equation*}
 using (A3) and Gronwall's inequality we obtain $z(\tau)=0$ for all $\tau \in [t_k,s_k)$  Therefore, system \eqref{nonlocal1} is T-Controllable over the interval $[t_k,s_k)$.
 
 Since, the system \eqref{nonlocal1} is T-controllable over the intervals $[0,t_1),[t_1,s_1), [s_k,t_{k+1})$ and $[t_k,s_k)$ for all $k$. Hence, the system is controllable over entire interval $[0,T_0]$. This completes the proof of the theorem.
  \end{proof}
 \begin{example}
 Let, $\mathcal{X}=L^2([0,\pi],\mathbb{R})$ and consider the system governed by non-instantaneous impulsive evolution equation:
 \begin{equation}
     \begin{aligned}
      \frac{\partial H(\tau,\psi)}{\partial \tau}&=\partial^{2}_{\psi} H(\tau,\psi)+F(\tau, H(\tau,\psi))+w_1(\tau,\psi)&\quad \tau \in [0,1/3)\cup [2/3,1], \\
      \frac{\partial H(\tau,\psi)}{\partial \tau}&=\partial^{2}_{\psi} H(\tau,\psi)+G(\tau,H(\tau,\psi))+w_2(\tau,\psi)&\quad \tau \in [1/3,2/3),\\
      H(\tau,0)&=0 \quad H(\tau, \pi)=0 & \quad \tau>0,\\
      H(0, \psi)&=H(\tau,\psi) \quad & \quad 0<\psi<\pi,
     \end{aligned}\label{nonexample1}
 \end{equation}
 over the interval $[0,1]$. Here, $H(\tau,\psi)$ is nonlocal opertor defined by $\sum^{n}_{i=1}\alpha_i H(\tau_i,\psi)$.
 
 Defining the operator on the space $\mathcal{X}$ as $A(\tau)=\partial^{2}_{\psi}$, $A(\tau)$ is the infinitesimal generator of the $C_0$ semigroup $\mathcal{T}(\tau)$. The representation of $\mathcal{T}(\tau)$ is \[\mathcal{T}(\tau)z=\sum^{\infty}_{m=0} \exp(\mu_m\tau)<z,\phi_m> \phi_m\]
 where, $\phi_m=\sqrt{2}sin(n\psi)$ for all $m=1,2,\cdots$ is the orthonormal basis corresponding to eigenvalue $\mu_m=-m^2$ of the operator $A$.
 
 With this concept the equation \eqref{nonexample1} can be rewritten as and abstract equation on the space $\mathcal{X}$ as
 \begin{equation}
    \begin{aligned}
     x'(\tau)&=A(\tau)+F(\tau,x)+W_1(\tau) &\quad  \tau \in [0,1/3)\cup [2/3,1],\\
     x(\tau)&=A(\tau)+G(\tau,x)+W_2(\tau) &\quad \tau \in [1/3,2/3),\\
     x(0)&=h(x), &\quad
     \end{aligned}\label{nonexample2}
 \end{equation}
 where, $x(\tau)=H(\tau,\cdot), W_1(\tau)=w_1(\tau,\psi), W_2(\tau)=w_1(\tau,\psi)$ and $h(x)=\sum^{n}_{i=1}\alpha_i x(\tau_i)$.
 The system \eqref{nonexample2} is trajectory controllable over the interval $[0,1]$ if the functions $F,G$ and $h$ satisfies the hypotheses of the theorem. 
 \end{example}
 \section{Conclusion}\label{conclusion}
 This article discussed the trajectory controllability of the system governed by non-instantaneous impulsive evolution equation having different perturbing force $F$ and $G$ with classical as well as non-local conditions on the Banach space. This type of systems are useful to many physical phenomena where only the perturbing forces were changes. Illustrations were also discussed to validate the derived results.

\end{document}